\documentclass[10pt]{amsart}

\usepackage{amsmath}
\usepackage{amssymb}
\usepackage{enumerate}
\usepackage{mathpazo}
\usepackage{amsthm}
\usepackage{tikz-cd}
\usepackage{mathrsfs}
\usepackage{amsrefs}
\usepackage{mathtools}
\usepackage[normalem]{ulem}
\usepackage{caption}
\usepackage{bigints}
\usepackage{parskip}
\usepackage{esint}
\usepackage{stackengine}
\stackMath

\usepackage{hyperref}
\hypersetup{
    colorlinks=true,
    linkcolor=blue,
    citecolor=cyan,
    filecolor=magenta,      
    urlcolor=cyan,
    pdftitle={Lattice Actions},
    pdfpagemode=FullScreen,
    }

\usepackage{geometry}
\usepackage{setspace,kantlipsum} 

\usepackage{thmtools} 
\usepackage{thm-restate} 
\usepackage{verbatim} 

\usepackage{tikz}
\usetikzlibrary{matrix,arrows}

\usepackage{graphicx}

\input{style.sty}

\title{Superdensity and Bounded Geodesics in Moduli Space}
\author{Josh Southerland}
\address{Indiana University, Bloomington, IN 47405}
\email{jwsouthe@iu.edu}

\begin{document}

 \begin{abstract}
  Following Beck-Chen, we say a flow $\phi_t$ on a metric space $(X, d)$ is \emph{superdense} if there is a $c > 0$ such that for every $x \in X$, and every $T>0$, the trajectory $\{\phi_t x\}_{0 \le t \le cT}$ is $1/T$-dense in $X$. We show that a linear flow on a translation surface is superdense if the associated Teichm\"uller geodesic is bounded. Conversely, if the linear flow is superdense, we show that along the Teichm\"uller geodesic, the diameter of the surface remains bounded. This generalizes work of Beck-Chen on \emph{lattice surfaces}, and is reminiscent of work of Masur on unique ergodicity.
 \end{abstract}

\setstretch{1.25}

\maketitle

 \section{Introduction}
 
 A \emph{translation surface} is an equivalence class of polygons (or finite sets of polygons) in the complex plane such that every side of a polygon is identified with a parallel side by translation. Two polygons (or sets of polygons) are equivalent provided we can cut one of the polygons and glue the pieces together to form the other polygon. Each translation surface has a global ``north" coming from the positive imaginary direction on the ambient plane. This direction is respected by the cutting and gluing equivalence relation. Translation surfaces are flat, away from a finite set of singular points. The singular points are cone points whose angles are integer multiples of $2\pi$~\cite{Wri15},~\cite{Zor06}. 
 
 There is a dynamical system commonly studied on translation surfaces: the linear flow $\Phi_t$ on the surface (Figure \ref{fig1}). This is the geodesic flow on the translation surface with the singular points removed. If a trajectory hits a singular point, we stop. 
 
  \begin{figure}[h]
   \centering
   \def\svgwidth{\columnwidth}
     \rotatebox{0}{\scalebox{0.3}{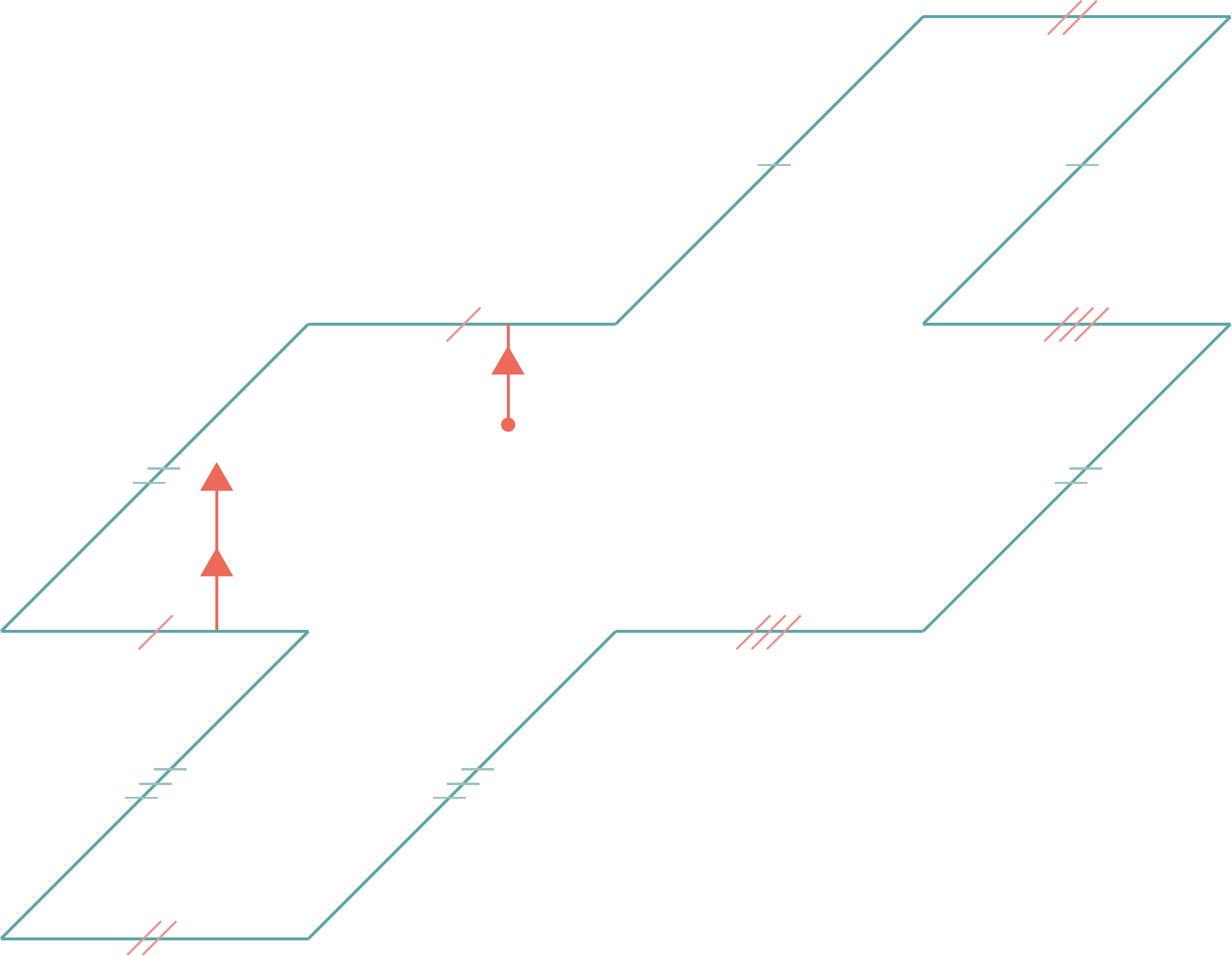}}
   \caption{Linear flow segment on a translation surface}
   \label{fig1} 
  \end{figure} 
 
 One motivation for studying such a system is its relationship to billiard trajectories on polygons with angles that are rational multiples of $\pi$. A billiard trajectory is a straight trajectory that ``bounces'' off the edges of the polygon following the rule that the angle of incidence is equal to the angle of reflection. Such polygons can be ``unfolded" to ``straighten" the billiard path. Since the angles are rational multiples of $\pi$, the number of directions a billiard path can follow is finite. After reflecting the polygon finitely many times, we arrive at only a single direction. The unfolded polygon is a translation surface, and the corresponding dynamical system is the linear flow~\cite{FK36},~\cite{KZ75},~\cite{Wri15},~\cite{Zor06}. 
 
 Equivalently, a \emph{translation surface} is a pair $(X,\omega)$ where $X$ is a compact, connected Riemann surface without boundary and $\omega$ a non-zero holomorphic differential on $X$. If we fix the genus of the underlying Riemann surface, the moduli space $\Omega_g$ of pairs $(X,\omega)$ forms a vector bundle over $\mathcal{M}_g$, the moduli space of genus $g$ Riemann surfaces, where the fiber over $X \in \mathcal{M}_g$ is the $g$-complex dimensional vector space $\Omega(X)$ of holomorphic $1$-forms on $X$. We will suppress the notation of the underlying Riemann surface and use the notation $\omega$ to denote a translation surface.
 
 The moduli space of translation surfaces is equipped with an $SL_2(\R)$-action, where the action is the usual linear action (in period coordinates). Elements of the form $g_t = \begin{bmatrix} e^{t} & 0 \\ 0 & e^{-t} \end{bmatrix}$ for any $t \in \R$ form a one-parameter subgroup which we will refer to as the (Teichm\"uller) \emph{geodesic flow} (Figure \ref{fig2}). 
 
 \begin{figure}[h]
   \centering
   \def\svgwidth{\columnwidth}
     \rotatebox{0}{\scalebox{0.6}{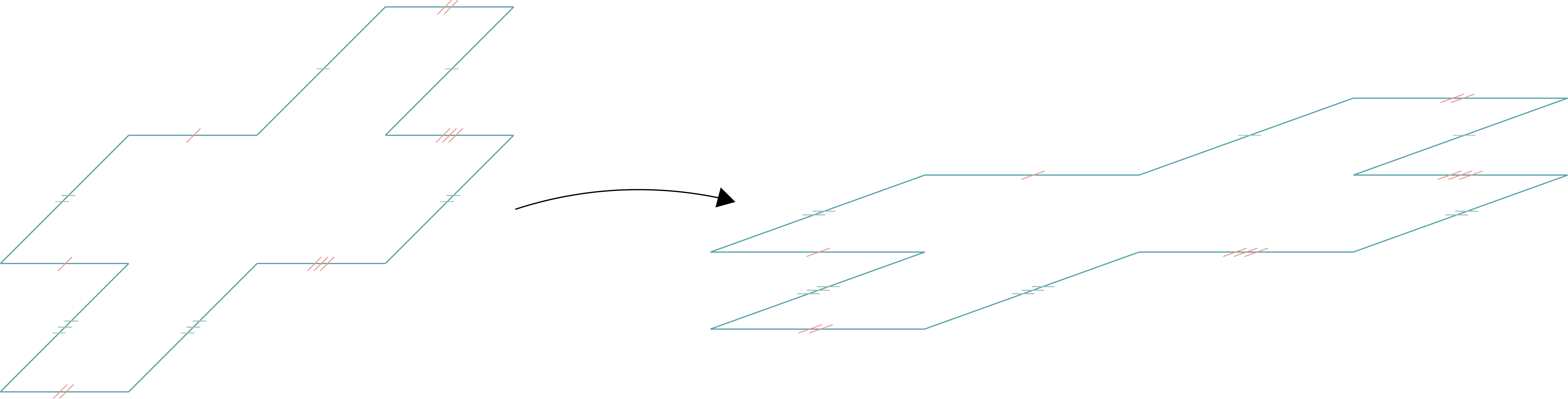}}
   \caption{Translation surface $\omega$ and $g_t \omega$}
   \label{fig2} 
  \end{figure} 
 
 There is a long history of interactions between the dynamical systems on individual translation surfaces and a dynamical system on the moduli space of translation surfaces, in particular, between the linear flow on a translation surface and geodesic flow on the moduli space.  Masur proved what is now known as Masur's Criterion by building on of earlier work with Kerckhoff and Smillie~\cite{KMS86},~\cite{M92},~\cite{MS91}. Masur used it as a tool to give an upper bound on the Hausdorff dimension of quadratic differentials whose vertical linear flow is not uniquely ergodic. 

 \begin{theorem*}[Masur's Criterion]
 Let $g_t$ denote the geodesic flow on the moduli space of translation surfaces and let $\omega$ be a translation surface. If $g_t \omega$ is non-divergent, that is, it returns to a compact set infinitely often, then the vertical straight line flow is uniquely ergodic. 
 \end{theorem*}

 We identify a quantitative density condition on the vertical flow of a translation surface $\omega$ that is equivalent to boundedness of the associated geodesic in moduli space. The condition is inspired by papers of Beck and Chen, where they study billiard trajectories on similar objects~\cite{BC21},~\cite{BC22}. 
 
 \begin{definition}[Superdensity]
 Let $\omega$ be a translation surface. We say that $\omega$ has \emph{superdense} linear flow $\Phi_t$ if there exists a constant $C > 0$ such that for every $T > 0$ where the flow is defined, the segment of the flow $\Phi_t$ for $t \in [0,T]$ is within $\frac{C}{T}$ to every point on $\omega$. Equivalently, the segment of the flow $\Phi_t$ for $t \in [0,CT]$ is $\frac{1}{T}$-dense on the surface. 
 \end{definition}
 
 Beck and Chen show that a linear flow on a square-tiled surface is superdense if and only if the slope in the associated direction is a badly-approximable number. 

 We give the following generalization.

 \begin{theorem}\label{thm:main}
 Let $\omega \in \Omega_g$ be a translation surface. The linear flow on $\omega$ is superdense if the associated Teichm\"uller geodesic $\{g_t\omega\}_{t \geq 0}$ is bounded in $\Omega_g$. Conversely, if $\omega$ has a superdense linear flow, then the diameter of $g_t\omega$ is bounded for all $t \geq 0$. 
 \end{theorem}

 We prove this result using the \emph{diameter} and \emph{dilatation} of the translation surface to control the quantitative density of the vertical (northward) linear flow. 
 
 As a corollary, since the diameter function is a proper function on the $SL_2(\R)$-orbit closures associated with a lattice surface, we have the following.

 \begin{corollary}
 Let $\omega \in \Omega_g$ be a lattice surface. The linear flow on $\omega$ is superdense if and only if the associated Teichm\"uller geodesic $\{g_t\omega\}_{t>0}$ is bounded in $\Omega_g$.
 \end{corollary}

 Moreover, in relation to Masur's criterion, we have the following. 
 
 \begin{corollary}
 If the linear flow on $\omega$ is superdense, it is uniquely ergodic. However, uniquely ergodic flows need not be superdense. 
 \end{corollary}

 \subsection{Related results}
 
 There have been a number of results that help explain the phenomenon described in Masur's criterion. For instance, Cheung and Masur constructed a half-translation surface (where we allow side identifications by translation and rotation by $\pi$) whose vertical flow is uniquely ergodic and the corresponding geodesic in the moduli space of Riemann surfaces diverges to infinity~\cite{CM06}. Not long after, Cheung and Eskin showed that if the geodesic diverges to infinity slowly enough, then the vertical linear flow is guaranteed to be uniquely ergodic~\cite{CE07}. 
 
 Similar to our result, Chaika and Trevi\~no found a closely related condition derived from the \emph{flat geometry} of the surface that implies unique ergodicity of the vertical linear flow on a translation surface~\cite{CT17},~\cite{T14}. Let $\delta(g_t \omega)$ be the systole on $g_t \omega$, by which we mean the shortest length of a non-contractible set of saddle connections. If $\int_0^{\infty} \delta^2(g_t \omega)\, dt$ diverges, then the vertical linear flow is uniquely ergodic. In short, the length of the shortest contractible set of saddle connections cannot get too short, too quickly. The geodesic must stay sufficiently far from the boundary of the moduli space (in some compact set) for a sufficient amount of time. 
 
 Our result differs in that, at least for lattice surfaces, it identifies a condition on the forward-time geodesic in moduli space that is equivalent to a quantitative density condition on the linear flow. Moreover, the work of Beck and Chen shows us that the slopes of these trajectories are badly approximable, meaning, superdense trajectories are as far as possible from (closed) trajectories of rational slopes. This seems to be a shadow of a philosophy being developed in homogeneous dynamics. Recently Lindenstrauss, Margulis, Mohammadi, and Shah gave effective bounds on time that the unipotent flow can spend near homogeneous subvarieties of an arithmetic quotient $G / \Gamma$~\cite{LMMS19}. This has become a tool for proving quantitative density statements about unipotent flows in this setting~\cite{LMW24}. Moreover, our results are akin to other results in homogeneous dynamics that seek to quantify the density of orbits. For example, the quantitative version of the Oppenheim conjecture seeks to give explicit quantitative information about the density of the orbits of unipotent flows~\cite{EMM98},~\cite{EMM05},~\cite{MM11}.
 
 
 \subsection{Acknowledgements}\label{acknowledgements}
 
 The author thanks Jayadev Athreya for proposing this question and providing guidance, and Matt Bainbridge, Jon Chaika, and Dami Lee for helpful discussions. The author would also like to thank the referee for helpful comments. 

 \section{Superdensity implies bounded diameter}

 \begin{lemma}\label{lem:superdense-implies-bounded} If the vertical (north or south) linear flow on $\omega$ is superdense, then the diameter of $g_t\omega$ is bounded for all $t \geq 0$.
 \end{lemma}

 \begin{proof} Let $\omega$ be such that the vertical linear flow $\Phi_s$ is superdense. Then, for any initial point on the surface, there exists a constant $C$ such that for any $T>0$, the vertical (north or south) segment $\gamma := \{\Phi_s(x) : s \in [0, CT]\}$ of length $CT$ is within $\frac{1}{T}$ of every point on $\omega$. Let $U = \{x \in \omega : \mathrm{dist}(x,\gamma) < \frac{1}{T}\}$ and note that $U$ is the entire surface. In other words, there exists a polygon in the plane $\C$ such that $\{ x \in \C: d(x,\gamma) < \frac{1}{T} \}$ covers the polygonal representation of the translation surface.

 Let $\tilde{t} \in [1, \infty)$ and apply $g_{\log(\tilde{t})}$ to the polygon. Notice that for every $\tilde{t}$, we have a cover of $g_{\log(\tilde{t})}\omega$. The diameter of $g_{\log(\tilde{t})}\omega$ is bounded by either $D=\sqrt{\left(\frac{cT}{\tilde{t}}\right)^2 + \left(\frac{2\tilde{t}}{T}\right)^2}$ or $D^{\prime} = \frac{cT}{\tilde{t}} + \frac{2}{T\tilde{t}}$. See Figure \ref{fig3}.
 
  \begin{figure}[h]
  \centering
  \def\svgwidth{\columnwidth}
    \rotatebox{0}{\scalebox{0.75}{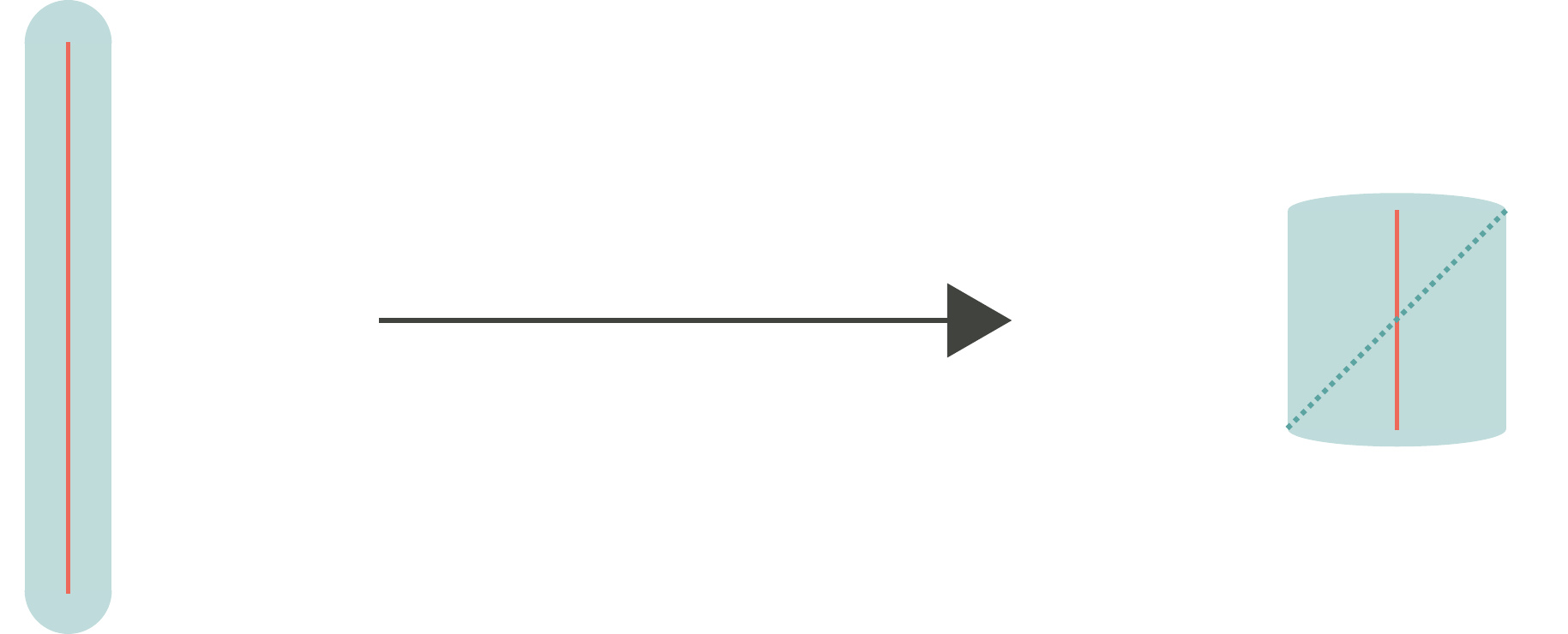}}
  \caption{Apply $g_{\log (\tilde{t})}$ to $U$}
  \label{fig3} 
 \end{figure}

 If $D \geq D^{\prime}$, we can pick $T = \frac{\sqrt{2}\tilde{t}}{\sqrt{c}}$ and note that the diameter is bounded by $4c$. If $D^{\prime} > D$, we can find a $T > 0$ such that the diameter is bounded by $c+2$.

 Now, let $t = \log{\tilde{t}}$, and our argument shows that for any $t > 0$, the diameter is bounded by $\max\{4c,c+2\}$. 
 \end{proof}

 Lemma \ref{thm:superdense-implies-bounded} is strict in the following sense: there exists a translation surface whose diameter remains bounded along the corresponding Teichm\"uller geodesic, but the geodesic leaves every compact set. In other words, superdensity on a general translation surface is not equivalent to boundedness of the forward-time Teichm\"uller geodesic in the moduli space. 

 Consider the following slit-torus construction. Take any two square tori, and rotate each torus so that the geodesic flow is recurrent. (Pick a direction associated with the contracting eigendirection of an affine diffeomorphism.) In the vertical direction of each torus, cut a small slit, small enough that when we apply the affine diffeomorphism, and then cut and reglue, we do not need to cut the slit. Glue the two tori together along these slits. Rescale the slit torus so that the total area is 1, and call the resulting translation surface $\omega$. See Figure \ref{fig4}. 

   \begin{figure}[h]
  \centering
  \def\svgwidth{\columnwidth}
    \rotatebox{0}{\scalebox{0.75}{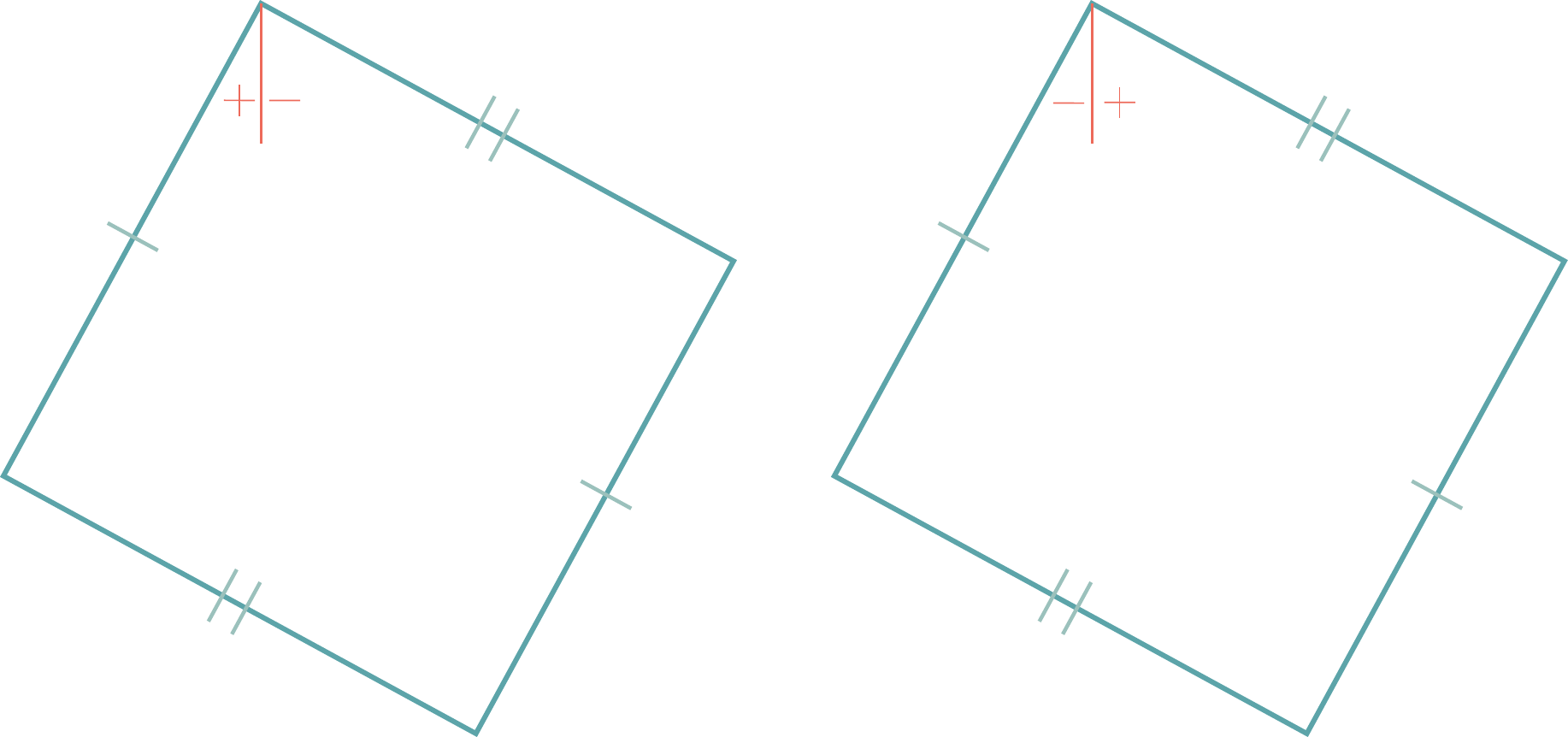}}
  \caption{Slit-torus construction}
  \label{fig4} 
 \end{figure}
 
 Flowing along the geodesic flow from this surface has the effect of warping each of the tori independently, but periodically returning to the square torus on each side of the slit. However, the slit itself contracts so that in the limit, the slit becomes a point and the surface degenerates to a wedge product of two tori.

 \section{Bounded in the moduli space implies superdensity}

 \begin{lemma}\label{lem:bounded-in-moduli-space-implies-superdense}
 Let $\omega$ denote a translation surface and $g_t$ the geodesic flow. If there exists a compact set $K$ such that $g_t \omega \in K$ for all $t>0$, then the vertical (north or south) linear flow on $\omega$ is superdense. 
 \end{lemma}

 \begin{proof} 
 Since $\{g_t \omega\}_{t\geq 0}$ is contained in a compact set in $\Omega_g$, the (continuous) projection of this set in $\mathcal{M}_g$ is compact. Let $X_t$ be the projection of $g_t\omega$. 
 
 Let $\mathcal{T}_g$ be the universal cover (Teichm\"uller space of genus $g$ Riemann surfaces) of $\mathcal{M}_g$, and choose a section $\sigma: \mathcal{M}_g \to \mathcal{T}_g$. Then, for any $t$, $\sigma(X_0)$ and $\sigma(X_t)$ are bounded distance from each other in $\mathcal{T}_g$. This means that for all $t$, $d_T(\sigma(X_0), \sigma(X_t))$ is bounded, where $d_T$ denotes the Teichm\"uller metric. As a consequence, the dilatation of the Teichm\"uller map between $X_0$ and $X_t$ is bounded for all $t$. Let $\phi_t: X_t \to X_0$ denote the inverse of that diffeomorphism, and let $K_{\omega}$ be the bound on the dilatation of this diffeomorphism. Note that in the flat coordinates induced by $\omega$ and $\omega_t$, the map $\phi_t$ stretches line segments in any direction less than $\sqrt{K_{\omega}}$.    

 
 Now, let $\gamma \subset \omega$ be a length $L$ segment of the vertical linear flow which avoids cone points. Let $D$ be the diameter of $\omega$. Fix some $\varepsilon > 0$ and a subsegment of $\gamma$ called $\gamma_{\varepsilon}$ of length $\varepsilon L$, where $\gamma_{\varepsilon}$ is centered in $\gamma$ (both share the same midpoint).  
 
 Apply $g_t$ for $t = \log\left(\frac{L\sqrt{K_{\omega}}}{D}\right)$. The length of the segment $g_t \gamma \subset g_t \omega$ is $e^{-t}L$. 

 Since the dilatation of $\phi_t$ is bounded by $K_{\omega}$, we know that the set $\{ x \in X_t: d(x,g_t \gamma_{\varepsilon}) < \frac{1}{2} D \sqrt{K_{\omega}}\}$, where $d$ is the flat distance, is the entire surface. In other words, there exists a polygon in the plane $\C$ such that $\{ x \in \C: d(x,g_t\gamma) < \frac{1}{2} D \sqrt{K_{\omega}}\}$ covers the polygonal representation of the translation surface. 

 Apply $g_t^{-1}$ to this polygon. 

 \begin{figure}[h]
  \centering
  \def\svgwidth{\columnwidth}
    \rotatebox{0}{\scalebox{0.6}{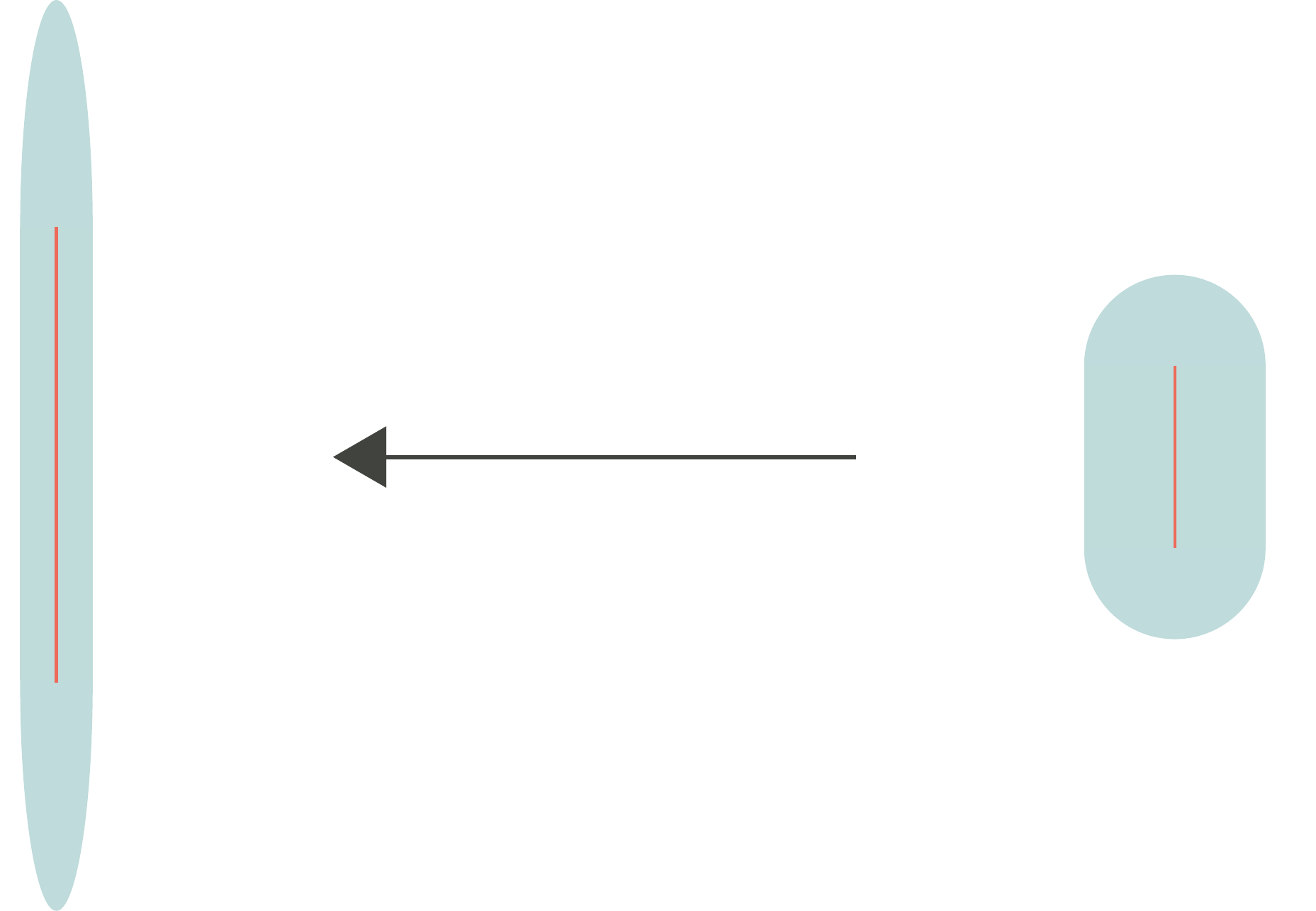}}
  \caption{Apply $g_{-t}$ to polygon}
  \label{fig5} 
 \end{figure} 
 
 The linear segment $\gamma$ returns to its original length, L, and the maximal width of the polygon becomes $$e^{-t} \cdot \frac{1}{2}\sqrt{K_{\omega}}D = \frac{KD^2}{2L} \left(\frac{1}{1-\varepsilon}\right).$$ Moreover, the maximal vertical length of the polygon is $L$. Indeed, 

 \begin{align*}
 \varepsilon L + \frac{1}{2}e^t\sqrt{K_{\omega}}D + \frac{1}{2}e^t\sqrt{K_{\omega}}D &= \varepsilon L + e^t\sqrt{K_{\omega}}D \\
 &= \varepsilon L + (1-\varepsilon)L \\
 &= L \text{.}
 \end{align*}

 Then, the segment $L$ is within $\frac{KD^2}{2L} \left(\frac{1}{1-\varepsilon}\right)$ of every point on the surface. We can pick $C = \frac{K_{\omega}D^2}{2}\cdot \frac{1}{1 - \varepsilon}$. Since the segment was arbitrary, we have that any length $L$ segment is within $\frac{C}{L}$ of every point on the surface, as desired. 
 \end{proof}

 A posteriori, we see that we can choose $C$ strictly larger than half of the square of the diameter of $\omega$ times $K_{\omega}$, the bound on the dilatation. 

 Lemma \ref{lem:superdense-implies-bounded} and Lemma \ref{lem:bounded-in-moduli-space-implies-superdense} imply the Theorem \ref{thm:main}.


\begin{bibdiv}
 \begin{biblist}


 \bib{BC21}{article}{
   title={Generalization of a density theorem of Khinchin and diophantine approximation},
   author={Beck, J.},
   author={Chen, W. W. L.},
   journal={J. Théor. Nombres Bordx.},
   volume={35},
   number={2}
   pages={511-542},
   date={2023}
   }
   
 \bib{BC22}{article}{
   title={Super-fast spreading of billiard orbits in rational polygons and geodesics on translation surfaces},
   author={Beck, J.},
   author={Chen, W.},
   journal={preprint},
   date={2022}
   }
   
 \bib{CT17}{article}{
   title={Logarithmic laws and unique ergodicity},
   author={Chaika, J.},
   author={Trevi\~no, R.},
   journal={Journal of Modern Dynamics},
   volume={11 (1)},
   pages={563-588},
   date={2017}
   }

 \bib{CE07}{article}{
   title={Unique Ergodicity of Translation Flows},
   author={Cheung, Y.},
   author={Eskin, A.},
   journal={Fields Institute Communications},
   volume={51},
   pages={213-222},
   date={2007}
   }
   
\bib{CM06}{article}{
   title={A divergent Teichm\"uller geodesic with uniquely ergodic vertical foliation},
   author={Cheung, Y.},
   author={Masur, H.},
   journal={Israel Journal of Mathematics},
   volume={152(1)},
   pages={1-15},
   date={2006}
   }

\bib{EMM98}{article}{
   title={Upper bounds and asymptotics in a quantitative version of the Oppenheim conjecture},
   author={Eskin, A.},
   author={Margulis, G.},
   author={Mozes, S.},
   journal={Annals of Mathematics},
   volume={147},
   pages={93-141},
   date={1998}
   }
   
\bib{EMM05}{article}{
   title={Quadratic forms of signature (2,2) and eigenvalue spacings on regular 2-tori},
   author={Eskin, A.},
   author={Margulis, G.},
   author={Mozes, S.},
   journal={Annals of Mathematics},
   volume={161},
   pages={679-725},
   date={2005}
   }
   

\bib{FK36}{article}{
   title={Concerning the transitive properties of geodesics on a rational polyhedron},
   author={Fox, R.},
   author={Kershner, R.},
   journal={Duke Math. J.},
   volume={2},
   pages={147-150},
   date={1936}
   }

\bib{KZ75}{article}{
   title={Topological transitivity of billiards in polygons},
   author={Katok, A.},
   author={Zemlyakov, A.},
   journal={Mat. Zametki},
   volume={18},
   pages={291-300},
   date={1975}
   }

\bib{KMS86}{article}{
   title={Interval exchange transformations and measured foliations},
   author={Kerckhoff, S.},
   author={Masur, H.},
   author={Smillie, J.},
   journal={Annals of Mathematics},
   volume={124},
   pages={293-311},
   date={1986}
   }

   
\bib{LMMS19}{article}{
   title={Quantitative behavior of unipotent flows and an effective avoidance principle},
   author={Lindenstrauss, E.},
   author={Margulis, G.},
   author={Mohammadi, A.},
   author={Shah, N.},
   journal={JAMA},
   pages={1-61},
   date={2023}
   }

\bib{LMW24}{article}{
   title={Effective equidistribution for some one parameter unipotent flows},
   author={Lindenstrauss, E.},
   author={Mohammadi, A.},
   author={Wang, Z.},
   journal={preprint},
   date={2024}
   }

\bib{MM11}{article}{
   title={Quantitative version of the Oppenheim conjecture for inhomogeneous quadratic forms},
   author={Margulis, G.},
   author={Mohammadi, A.},
   journal={Duke Mathematical Journal},
   volume={158(1)},
   pages={121-160},
   date={2011}
   }
   

\bib{M92}{article}{
   title={Hausdorff dimension of the set of nonergodic folitations of a quadratic differential},
   author={Masur, H.},
   journal={Duke Mathematical Journal},
   volume={66, no.3},
   pages={387-442},
   date={1992}
   }
   
\bib{MS91}{article}{
   title={Hausdorff Dimension of sets of nonergodic measured foliations},
   author={Masur, H.},
   author={Smillie, J.},
   journal={Annals of Mathematics},
   volume={134},
   pages={455-543},
   date={1991}
   }
   
\bib{T14}{article}{
   title={On the ergodicity of flat surfaces of finite area},
   author={Trevi\~no, R.},
   journal={Geometric and Functional Analysis},
   volume={24},
   pages={360-386},
   date={2014}
   }
   
   
\bib{Wri15}{article}{
   title={Translation Surfaces and their Orbit Closures},
   subtitle={An Introduction for a Broad Audience},
   author={Wright, A.},
   journal={EMS Surv. Math. Sci.},
   date={2015}
   }
   
 \bib{Zor06}{article}{
   title={Flat Surfaces},
   author={Zorich, A.},
   journal={Frontiers in Number Theory, Physics, and Geometry},
   volume={Vol. 1},
   date={2006}
   }

 \end{biblist}
 \end{bibdiv}
\end{document}